\documentclass[11pt]{article}

\usepackage{amsmath}
\usepackage{amsthm}
\usepackage{amsfonts}
\usepackage{bbm}

\newcommand{\mmp}{\mathbb{P}}

\newcommand{\od}{\overset{d}{=}}

\newcommand{\me}{\mathbb{E}}
\newcommand{\mr}{\mathbb{R}}
\newcommand{\mn}{\mathbb{N}}
\newcommand{\1}{\mathbbm{1}}

\newcommand{\lix}{\underset{x\to\infty}{\lim}}

\newtheorem{thm}{Theorem}[section]

\newtheorem{assertion}[thm]{Proposition}
\theoremstyle{definition}

\newtheorem{ex}[thm]{Example}
\theoremstyle{remark}

\begin{document}

\title{Exponential Moments of First Passage Times and Related Quantities for Random Walks}
\date{\today}

\author{Alexander Iksanov   \\
                National T.\ Shevchenko University of Kiev  \\
                \and
                Matthias Meiners\footnote{Research supported by DFG-grant Me 3625/1-1}   \\
                Uppsala University}

\maketitle
\begin{abstract}
For a zero-delayed random walk on the real line, let $\tau(x)$,
$N(x)$ and $\rho(x)$ denote the first passage time into the
interval $(x,\infty)$, the number of visits to the interval
$(-\infty,x]$ and the last exit time from $(-\infty,x]$,
respectively. In the present paper, we provide ultimate criteria
for the finiteness of exponential moments of these quantities.
Moreover, whenever these moments are finite, we derive their
asymptotic behaviour, as $x \to \infty$.
\end{abstract}


\hphantom{2000 Mathematics Subject Classification: }Secondary: 60G40    

\section{Introduction and main results}

Let $(X_n)_{n \geq 1}$ be a sequence of i.i.d.\ real-valued random
variables and $X := X_1$. Further, let $(S_n)_{n \geq 0}$ be the
zero-delayed random walk with increments $S_n-S_{n-1} = X_n$, $n
\geq 1$. For $x\in\mr$, define the \emph{first passage time} into
$(x,\infty)$
\begin{equation*}
\tau(x) ~:=~    \inf\{n \in \mn_0: S_n>x\},
\end{equation*}
the \emph{number of visits} to the interval $(-\infty, x]$
\begin{equation*}
N(x)    ~:=~    \#\{n\in\mn: S_n \leq x\}   ~=~ \sum_{n \geq 1} \1_{\{S_n \leq x\}},
\end{equation*}
and the \emph{last exit time} from $(-\infty,x]$
\begin{gather*}
\rho(x) ~:=~
\left\{%
\begin{array}{ll}
    \sup\{n \in \mn: S_n \leq x\}, & \hbox{\text{if} \ $\inf_{n \geq 1} S_n \leq x$,} \\
    0, & \hbox{\text{if} \ $\inf_{n \geq 1} S_n >
x$.} \\
\end{array}%
\right.
\end{gather*}
Note that, for $x\geq 0$,
\begin{equation*}
\rho(x) ~=~	\sup\{n \in \mn_0: S_n \leq x\}.
\end{equation*}
For typographical ease, throughout the text we write $\tau$ for
$\tau(0)$, $N$ for $N(0)$ and $\rho$ for $\rho(0)$.

Our aim is to find criteria for the finiteness of the exponential
moments of $\tau(x)$, $N(x)$ and $\rho(x)$, and to determine the
asymptotic behaviour of these moments, as $x \to \infty$.

Assuming that $0 < \me X < \infty$, Heyde \cite[Theorem 1]{Heyde1964} proved that
\begin{equation*}
\me e^{a\tau(x)} < \infty \text{ for some } a>0 \qquad \text{iff}
\qquad \me e^{bX^-} < \infty \text{ for some } b>0.
\end{equation*}
See also \cite[Theorem 2]{Borov} and \cite[Theorem 2]{Doney1989} for relevant results.

When $\mmp\{X \geq 0\}=1$ and $\mmp\{X=0\}<1$,
\begin{equation}\label{xyz}
\tau(x)-1 ~=~ N(x) ~=~ \rho(x), \quad   x \geq 0.
\end{equation}
Plainly, in this case, criteria for all the three random variables
are the same (Proposition \ref{main0}). An intriguing consequence
of our results in case when $\mmp\{X<0\}\mmp\{X>0\}>0$, in which
\begin{equation}    \label{eq:tau_leq_N_leq_rho+1}
\tau(x)-1   ~\leq~  N(x) ~\leq~ \rho(x),
\quad   x \geq 0,
\end{equation}
is that provided the abscissas of convergence of the moment generating
functions of $\tau(x)$, $N(x)$ and $\rho(x)$ are positive there exists
a unique value $R>0$ such that {\it typically}
\begin{equation*}
\me e^{a\tau(x)}<\infty, \ \me e^{aN(x)}<\infty \ \text{iff} \
a \leq R, \ \text{yet} \ \me e^{a\rho(x)}<\infty \ \text{iff} \
a < R.
\end{equation*}
In particular, typically
\begin{equation*}
\me e^{R \tau(x)}<\infty, \ \me e^{R N(x)}<\infty,
\ \text{but} \ \me e^{R \rho(x)}=\infty.
\end{equation*}
Also we prove that whenever the exponential moments are finite
they exhibit the following asymptotics:
$$\me e^{a\tau(x)}~\sim~ C_1 e^{\gamma x}, \ \me e^{aN(x)}~\sim~ C_2 e^{\gamma x}, \ \me e^{a\rho(x)}~\sim~ C_3
e^{\gamma x}, \ \ x\to\infty,$$ for explicitly given $\gamma>0$
and distinct positive constants $C_i$, $i=1,2,3$ (when the law of $X$ is lattice with span $\lambda>0$ the limit is taken over $x\in
\lambda \mn$). Our results should be compared (or contrasted) to
the known facts concerning power moments (see \cite[Theorem 2.1
and Section 4.2]{KestMaller} and \cite[Theorem 2.2]{KestMaller},
respectively): for $p>0$
\begin{equation*}
\me (\tau(x))^{p+1}<\infty \quad   \Leftrightarrow \quad \me
(N(x))^p<\infty \quad   \Leftrightarrow \quad \me
(\rho(x))^p<\infty;
\end{equation*}
\begin{equation*}
\me (\tau(x))^p ~\asymp~ \me (N(x))^p~\asymp~ \me
(\rho(x))^p~\asymp~ \bigg(\frac{x}{\me \min (X^+, x)}\bigg)^p, \ \
x \to \infty
\end{equation*}
where $f(x)\asymp g(x)$ means that
$0<\underset{x\to\infty}{\lim\inf}\,\frac{f(x)}{g(x)}\leq
\underset{x\to\infty}{\lim\sup}\,\frac{f(x)}{g(x)} <\infty$.

Proposition \ref{main0} is due to Beljaev and Maksimov
\cite[Theorem 1]{BelMaks}. A shorter proof can be found in
\cite[Theorem 2.1]{IksMei}.

\begin{assertion}\label{main0}
Assume that $\mmp\{X\geq 0\}=1$ and let $\beta:=\mmp\{X=0\}\in [0,1)$.
Then for $a>0$ the following conditions are equivalent:
\begin{equation*}
\me e^{a\tau(x)}    <   \infty  \text{ for some (hence every)} \ x\geq 0;
\end{equation*}
\begin{equation*}
a<-\log \beta
\end{equation*}
where $-\log \beta := \infty$ if $\beta = 0$. The same equivalence
also holds for $N(x)$ and $\rho(x)$.
\end{assertion}

The following theorem provides sharp criteria for the finiteness
of exponential moments of $\tau(x)$ and $N(x)$ in the case when
$\mmp\{X<0\}>0$.

\begin{thm}\label{main1}
Let $a>0$ and $\mmp\{X<0\}>0$. Then the following conditions are
equivalent:
\begin{align}   \label{1211}
\sum_{n \geq 1} \frac{e^{an}}{n} \mmp\{S_n\leq x\}  ~<~ \infty
& \text{ for some (hence every) }   x\geq 0;    \\
\label{1212}
\me e^{a\tau(x)}    <   \infty
& \text{ for some (hence every) }   x\geq 0;    \\
\label{1213}
\me e^{aN(x)}   <   \infty
& \text{ for some (hence every) }   x\geq 0;    \\
\label{12121}
a \leq R & :=   -\log \inf_{t\geq 0} \, \me e^{-tX}.
\end{align}
\end{thm}

Our next theorem provides the corresponding result for the last exit time $\rho(x)$.

\begin{thm} \label{main2}
Let $a>0$ and $\mmp\{X<0\}>0$. Then the following conditions are
equivalent:
\begin{align}   \label{21211}
\sum_{n \geq 0} e^{an} \mmp\{S_n\leq x\}<\infty
& \text{ for some (hence every) }   x\geq 0;    \\
\label{1214}
\me e^{a\rho(x)}    <   \infty
& \text{ for some (hence every) }   x \geq 0;   \\
\label{212121} a < R = -\log \inf_{t\geq 0}\,\me e^{-tX} \quad &
\text{or}   \quad a=R \text{ and } \me X e^{-\gamma_0 X} > 0
\end{align}
where $\gamma_0$ is the unique positive number such that $\me
e^{-\gamma_0 X} = e^{-R}$.
\end{thm}

Now we turn our attention to the asymptotic behaviour of $\me
e^{a\tau(x)}$, $\me e^{aN(x)}$ and $\me e^{a\rho(x)}$ and start by
recalling a known result which, given in other terms, can be found
in \cite[Theorem 2.2]{IksMei}. In view of equality \eqref{xyz} we
only state it for $\me e^{a\tau(x)}$. The phrase `$X$
is $\lambda$-lattice' used in formulations of Proposition
\ref{pos} and Theorem \ref{Thm:asymptotics} is a shorthand for
`The law of $X$ is lattice with span $\lambda>0$'.
\begin{assertion}\label{pos}
Let $a>0$, $\mmp\{X\geq 0\}=1$ and $\mmp\{X=0\}<1$. Assume that
$\me e^{a\tau(x)}<\infty$ for some (hence every) $x\geq 0$. Then,
as $x \to \infty$,
    \begin{equation*}    \label{eq:posasymptotics_Ee^atau(x)}
    \me e^{a \tau(x)}   ~\sim~ e^{\gamma x} \, \times \,
    \begin{cases}
        \frac{1-e^{-a}}{\gamma \me Xe^{-\gamma X}},
        &   \text{if $X$ is non-lattice,}   \\
         \frac{\lambda(1-e^{-a})}{(1-e^{-\lambda\gamma})\me Xe^{-\gamma
         X}},
        &   \text{if $X$ is $\lambda$-lattice}
    \end{cases}
    \end{equation*}
where $\gamma$ is a unique positive number such that $\me
e^{-\gamma X}=e^{-a}$, and in the $\lambda$-lattice case the limit
is taken over $x \in \lambda \mn$.
\end{assertion}

Let $$\varphi: [0,\infty)~\to ~ (0,\infty], \ \ \varphi(t):=\me
e^{-tX}$$ be the Laplace transform of $X$. When $0 < a \leq R$ and
$\mmp\{X < 0\}
> 0$, there exists a minimal $\gamma>0$\label{222} such that
$\varphi(\gamma) = e^{-a}$. This $\gamma$ can be used to define a
new probability measure $\mmp_\gamma$ by
\begin{equation}    \label{mea}
\me_\gamma h(S_0,\ldots, S_n)
~=~ e^{an} \me e^{-\gamma S_n}h(S_0,\ldots, S_n),
\quad n\in\mn,
\end{equation}
for each nonnegative Borel function $h$ on $\mr^{n+1}$, where $\me_{\gamma}$ denotes expectation with respect to $\mmp_{\gamma}$. Since $\me_{\gamma} X = \me_{\gamma} S_1 = -e^a \varphi'(\gamma)$ (where $\varphi'$ denotes the left derivative of $\varphi$) and since $\varphi$ is decreasing and convex on $[0,\gamma]$, there are only two possibilities:
\begin{equation}    \label{eq:E_{gamma}X}
\text{Either}   \quad   \me_\gamma X \in (0,\infty)
\quad   \text{or}   \quad
\me_\gamma X=0.
\end{equation}
When $a<R$, then the first alternative in \eqref{eq:E_{gamma}X}
prevails. When $a=R$, then typically $\varphi'(\gamma)=0$ since
$\gamma$ is then unique minimizer of $\varphi$ on $[0,\infty)$. In
particular, $\me_{\gamma} X = 0$. But even if $a=R$ it can occur
that $\me_{\gamma} X > 0$ or, equivalently, $\varphi'(\gamma) <
0$. Of course, then $\gamma$ is the right endpoint of the interval
$\{t\geq 0: \varphi(t)<\infty\}$. We provide an example of this
situation in Section \ref{sec:apps_and_exs}.

Now we are ready to formulate the last result of the paper.
\begin{thm} \label{Thm:asymptotics}
Let $a>0$ and $\mmp\{X<0\}>0$.

\begin{itemize}
    \item[(a)] Assume that
$\me e^{a\tau(x)}<\infty$ for some (hence every) $x\geq 0$. Then
$\me_{\gamma} S_{\tau}$ is positive and finite, and, as $x \to
\infty$,
    \begin{equation}    \label{eq:asymptotics_Ee^atau(x)}
    \me e^{a \tau(x)}   ~\sim~ e^{\gamma x} \, \times \,
    \begin{cases}
        \frac{\me (e^{a\tau} - 1)}{\gamma \me_{\gamma} S_{\tau}},
        &   \text{if $X$ is non-lattice,}   \\
        \frac{\lambda \me (e^{a\tau} - 1)}{(1-e^{-\lambda\gamma}) \me_{\gamma}
        S_{\tau}},
        &   \text{if $X$ is $\lambda$-lattice.}
    \end{cases}
    \end{equation}
       \item[(b)]
       Assume that
$\me e^{aN(x)}<\infty$ for some (hence every) $x\geq 0$. Then
$\me_{\gamma} S_{\tau}$ is positive and finite, and, as $x \to
\infty$,
    \begin{equation}    \label{eq:asymptotics_Ee^aN(x)}
    \me e^{a N(x)}   ~\sim~ e^{\gamma x} \, \times \,
    \begin{cases}
        \frac{e^{-a} \me_\gamma \int_0^{S_\tau} e^{\gamma y}\me[e^{aN(-y)}] \, {\rm d}y}{\me_\gamma S_\tau},
        &   \text{if $X$ is non-lattice,}   \\
        \frac{\lambda e^{-a} \me_\gamma \sum_{k=1}^{S_\tau/\lambda}
        e^{\gamma \lambda k}\me [e^{aN(-\lambda k)}]}{\me_\gamma S_\tau},
        &   \text{if $X$ is $\lambda$-lattice.}
    \end{cases}
    \end{equation}
    \item[(c)] Assume that
$\me e^{a\rho(x)}<\infty$ for some (hence every) $x\geq 0$. Then
$M := \inf_{n \geq 1} S_n$ is positive with positive probability,
and, as $x \to \infty$,
    \begin{multline}    \label{eq:asymptotics_Ee^arho(x)}
\me e^{a \rho(x)}   ~\sim~ e^{\gamma x} \, \times \,
\begin{cases}
\frac{e^{-a}(1-\me e^{-\gamma M^+})}{\gamma \me X e^{-\gamma X}},
&
\text{if $X$ is non-lattice,}   \\
\frac{\lambda e^{-a}(1-\me e^{-\gamma
M^+})}{(1-e^{-\lambda\gamma}) \me X e^{-\gamma X}}, & \text{if $X$ is $\lambda$-lattice.}
\end{cases}
\end{multline}
\end{itemize}
In the $\lambda$-lattice case the limit is taken over $x \in
\lambda \mn$.
\end{thm}

The rest of the paper is organized as follows. Section
\ref{sec:proofs} is devoted to the proofs of Theorems \ref{main1},
\ref{main2} and \ref{Thm:asymptotics}. In Section
\ref{sec:apps_and_exs} we provide three examples illustrating our
main results.

\section{Proofs of the main results}    \label{sec:proofs}

\begin{proof}[Proof of Theorem \ref{main1}]
\eqref{12121} $\Rightarrow$ \eqref{1211}.
Pick any $a\in (0,R]$ and let $\gamma$ be as defined on
p.~\pageref{222}. With this $\gamma$, the equality
\begin{equation*}
Z_\gamma(A) ~:=~    \sum_{n\geq1} \frac{\mmp_\gamma\{S_n\in A\}}{n}
\end{equation*}
where $A\subset \mr$ is a Borel set, defines a measure which is
finite on bounded intervals. Furthermore, according to
\cite[Proposition 1.1 and Theorem 1.2]{Als}, if $\me_\gamma X>0$
then $Z_\gamma((-\infty, 0])<\infty$, whereas if $\me_\gamma X=0$
(this may only happen if $a=R$), then the function $x\mapsto
Z_\gamma((-x,0])$, $x>0$, is of sublinear growth. Hence, for every
$x\geq 0$,
\begin{equation*}
\sum_{n \geq 1} \frac{e^{an}}{n} \mmp\{S_n\leq x\} ~=~ \sum_{n
\geq 1} \frac{1}{n} \me_\gamma e^{\gamma S_n} \1_{\{S_n\leq x\}}
~=~ \underset{(-\infty,x]}{\int} \!\!\! e^{\gamma y} \,
Z_\gamma({\rm d}y) ~<~ \infty.
\end{equation*}

\noindent \eqref{1211} $\Rightarrow$ \eqref{12121}. Suppose
\eqref{1211} holds for some $x=x_0\geq 0$ and $a>R$. Pick
$\varepsilon\in (0,a-R)$. Then $\sum_{n \geq
0}e^{(a-\varepsilon)n} \mmp\{S_n \leq x_0\}<\infty$ which is a
contradiction to \cite[Theorem 2.1(aiii)]{IksMei} (reproduced here
as equivalence \eqref{21211} $\Leftrightarrow$ \eqref{212121} of
Theorem \ref{main2}).

\noindent \eqref{1211} $\Rightarrow$ \eqref{1212}. The argument
given below will also be used in the proof of Theorem
\ref{Thm:asymptotics}.

If \eqref{1211} holds for some $x \geq 0$ then, according to the
already proved equivalence \eqref{1211} $\Leftrightarrow$
\eqref{12121}, first, $a \leq R$ and, secondly, \eqref{1211} holds
for every $x \geq 0$. For $0 < a \leq R$ and $x \geq 0$, we have
\begin{eqnarray}
\me e^{a\tau(x)}
& = &
1 + (e^a-1) \sum_{n \geq 0} e^{an} \mmp\{\tau(x) > n\}  \notag  \\
& = &
1 + (e^a-1) \sum_{n \geq 0} e^{an} \mmp\{M_n \leq x\}  \label{eq:Ee^atau(x)}
\end{eqnarray}
where $M_n:=\max_{0\leq k\leq n} S_k$, $n \in \mn_0$.
According to \cite[Formula (2.9)]{Doney1989},
\begin{equation}
\sum_{n \geq 0} e^{an} \mmp\{M_n \leq x\}~ = ~
\frac{\me e^{a\tau} - 1}{e^a-1} \, \sum_{j \geq 0} e^{aj}
\mmp\{L_j = j, S_j \leq x\}  \label{eq:Doney}
\end{equation}
where $L_j = \inf\{i \in \mn_0: S_i = M_j\}$, $j \in \mn_0$.
Since $a \leq R$, we can use the exponential measure
transformation introduced in \eqref{mea}, which gives
\begin{equation*}
e^{aj} \mmp\{L_j = j, S_j \leq x\} ~=~ \me_{\gamma} e^{\gamma S_j} \1_{\{L_j = j, S_j \leq x\}}.
\end{equation*}
Observe that $L_j = j$ holds iff $j = \sigma_k$ for some
$k\in\mn_0$ where $\sigma_k$ $(\sigma_0:=0)$ denotes the $k$th
strictly ascending ladder epoch of the random walk $(S_n)_{n \geq
0}$. Thus,
\begin{align}
\sum_{j\geq0} e^{aj} & \mmp\{L_j = j, S_j \leq x\}
~=~ \sum_{j\geq0} \me_{\gamma} e^{\gamma S_j} \1_{\{L_j = j, S_j \leq x\}}  \notag  \\
& =~
\sum_{j\geq0} \me_{\gamma} \sum_{k \geq 0} e^{\gamma S_{\sigma_k}} \1_{\{\sigma_k = j, S_{\sigma_k} \leq x\}}
~=~ \me_{\gamma} \sum_{k \geq 0} e^{\gamma S_{\sigma_k}} \1_{\{S_{\sigma_k} \leq x\}}   \notag  \\
& =~ e^{\gamma x} \int_\mr e^{-\gamma(x-y)} \1_{[0,\infty)}(x-y)
\, U^>_{\gamma}({\rm d}y) ~=:~    e^{\gamma x} Z^>_\gamma(x)
\label{eq:Z(x)}
\end{align}
where $U^>_{\gamma}$ denotes the renewal function of the random
walk $(S_{\sigma_k})_{k \geq 0}$ under $\mmp_{\gamma}$, that is,
$U^>_{\gamma}(\cdot) = \sum_{k \geq 0} \mmp_\gamma\{S_{\sigma_k}
\in \cdot\}$. Thus, $Z^>_\gamma(x)$ is finite for all $x \geq 0$
since it is the integral of a directly Riemann integrable
function with respect to $U^>_{\gamma}$.

\noindent
\eqref{1212} $\Rightarrow$ \eqref{1211} and \eqref{1213} $\Rightarrow$ \eqref{1211}.
Since $\tau(y) \leq N(y)+1$, $y \geq 0$, it suffices to prove the first implication.
To this end, let
\begin{equation*}    \label{eq:K(a)}
K(a)    ~:=~    \sum_{n \geq 1} \frac{e^{an}}{n} \mmp\{S_n \leq 0\}.
\end{equation*}
By a generalization of Spitzer's formula \cite[Formula
(2.6)]{Doney1989}, the assumption $\me e^{a\tau} < \infty$
immediately entails the finiteness of $K(a)$:
\begin{equation*}
\infty  ~>~ \me e^{a\tau}   ~=~ 1+(e^a-1) \sum_{n \geq 0} e^{an} \mmp\{M_n=0\}
~=~ 1+(e^a-1)e^{K(a)}.
\end{equation*}
We already know that if the series in \eqref{1211} converges for $x=0$,
\textit{i.e.}, if $K(a)<\infty$, then it converges for every $x \geq 0$.

\noindent
\eqref{1212} $\Rightarrow$ \eqref{1213}.
By the equivalence \eqref{1211} $\Leftrightarrow$ \eqref{1212},
$\me e^{a\tau(x)}<\infty$ for every $x\geq 0$. According to
\cite[Formula (3.54)]{KestMaller},
\begin{equation*}
\mmp\{N=k\}  ~=~ \mmp\{\inf_{n \geq 1} S_n>0\} \mmp\{\tau>k\}, \ \
k\in\mn_0,
\end{equation*}
where $\mmp\{\inf_{n \geq 1} S_n>0\}>0$, since, under the present
assumptions, $(S_n)_{n\geq 0}$ drifts to $+\infty$ a.s. Hence,
$\me e^{aN}<\infty$. Further, for $y\in\mr$,
\begin{equation}    \label{eq:hatN}
\widehat{N}(x,y)    ~:=~    \sum_{n>\tau(x)} \1_{\{S_n-S_{\tau(x)} \leq y\}}
\end{equation}
is a copy of $N(y)$ that is independent of $(\tau(x), S_{\tau(x)})$. We have
\begin{equation}    \label{eq:N_decomposed}
N(x)    ~=~ \tau(x)-1 + \widehat{N}(x,x-S_{\tau(x)})    ~\leq~ \tau(x) + \widehat{N}(x,0)
\end{equation}
Hence, $\me e^{aN(x)} < \infty$, for every $x\geq 0$. The proof
is complete.
\end{proof}

\begin{proof}[Proof of Theorem \ref{main2}]
The equivalence \eqref{21211} $\Leftrightarrow$ \eqref{212121} has been proved
in \cite[Theorem 2.1]{IksMei}.

\noindent
\eqref{21211} $\Rightarrow$ \eqref{1214}.
According to the just mentioned equivalence,
if \eqref{21211} holds for some $x\geq 0$
it holds for every $x\geq 0$. It remains to note that for $x \geq 0$
\begin{equation}    \label{eq:rho(x)=n}
\mmp\{\rho(x)=n\}
~=~ \underset{(-\infty,x]}{\int} \! \mmp\{\inf_{k \geq 1} S_k > x-y\} \, \mmp\{S_n \in {\rm d}y\}
~\leq~  \mmp\{S_n\leq x\}.
\end{equation}

\noindent \eqref{1214} $\Rightarrow$ \eqref{212121}. Suppose $\me
e^{a\rho(x_0)} < \infty$ for some $x_0\geq 0$ and $a>0$. Since
$\me e^{a\rho(x)}$ is increasing in $x$, we have $\me e^{a\rho} <
\infty$. Condition $a \leq R$ must hold in view of
\eqref{eq:tau_leq_N_leq_rho+1} and implication \eqref{1212}
$\Rightarrow$ \eqref{12121} of Theorem \ref{main1}. If $a<R$, we
are done. In the case $a=R$ it remains to show that
\begin{equation}\label{2222}
\me Xe^{-\gamma_0 X} > 0.
\end{equation}
Define the measure $V$ by
\begin{equation}    \label{eq:V(A)}
V(A)    ~:=~    \sum_{n \geq 0} e^{Rn} \mmp\{S_n\in A\},
\end{equation}
for Borel sets $A \subset \mr$. Then from \eqref{eq:rho(x)=n} we infer that
\begin{equation}    \label{eq:Ee^arho(x)}
\infty  ~>~ \me e^{R\rho} ~=~ \underset{(-\infty,0]}{\int} \!
\mmp\{\inf_{n \geq 1} S_n > -y\} \, V({\rm d}y).
\end{equation}
Under the present assumptions, the random walk $(S_n)_{n \geq 0}$
drifts to $+\infty$ a.s. Thus, $\mmp\{\inf_{n
\geq 1} S_n > \varepsilon\} > 0$ for some $\varepsilon > 0$.
With such an $\varepsilon$,
\begin{equation*}
\infty ~>~
\underset{(-\varepsilon,0]}{\int} \! \mmp\{\inf_{n \geq 1} S_n > -y\} \, V({\rm d}y)
~\geq~  \mmp\{\inf_{n \geq 1} S_n > \varepsilon\} V((-\varepsilon,0]).
\end{equation*}
Therefore,
\begin{eqnarray*}
\infty &>&
V((-\varepsilon,0]) ~=~ \sum_{n=0}^\infty\me_{\gamma_0}
e^{\gamma_0 S_n}\1_{\{-\varepsilon < S_n \leq 0\}}	\\
& \geq &
e^{-\gamma_0 \varepsilon} \sum_{n=0}^\infty
\mmp_{\gamma_0}\{-\varepsilon < S_n \leq 0\}.
\end{eqnarray*}
Hence $(S_n)_{n \geq 0}$ must be transient under
$\mmp_{\gamma_0}$, which yields the validity of \eqref{212121} in
view of \eqref{eq:E_{gamma}X} and $\me_{\gamma_0} S_1 = e^R \me X
e^{-\gamma_0 X}$. The proof is complete.
\end{proof}

\begin{proof}[Proof of Theorem \ref{Thm:asymptotics}]
(a) In view of \eqref{eq:Ee^atau(x)}, \eqref{eq:Doney} and
\eqref{eq:Z(x)}, in order to find the asymptotics of $\me e^{a
\tau(x)}$, it suffices to determine the asymptotic behaviour of
$Z^>_\gamma(x)$ defined in \eqref{eq:Z(x)}. By the key renewal
theorem on the positive half-line,
\begin{equation} Z^>_\gamma(x)
~\underset{x \to \infty}{\to}~
\begin{cases}
\frac{1}{\gamma \me_{\gamma} S_{\tau}}
&   \text{if $X$ is non-lattice,}   \\
\frac{\lambda}{(1-e^{-\lambda\gamma}) \me_{\gamma} S_{\tau}} &
\text{if $X$ is $\lambda$-lattice}
\end{cases}
\end{equation}
where the limit $x \to \infty$ is taken over $x \in \lambda \mn$
when $X$ is lattice with span $\lambda > 0$.

It remains to check that $\me_{\gamma} S_{\tau}$ is finite. As
pointed out in \eqref{eq:E_{gamma}X}, either $\me_{\gamma} X \in
(0,\infty)$ or $\me_{\gamma} X = 0$. In the first case, $S_n \to
\infty$ a.s.\ under $\mmp_{\gamma}$ and, therefore, $\me_{\gamma}
\tau < \infty$, see, for instance, \cite[Theorem 2,
p.\,151]{ChowTeicher}, which yields $\me_{\gamma} S_{\tau} <
\infty$ by virtue of Wald's identity. If, on the other hand,
$\me_{\gamma} X = 0$, then $\me_{\gamma} \tau = \infty$ and we
cannot argue as above. But in this case, by \cite[Formula
(4a)]{Doney1980}, $\me_{\gamma} (S_1^+)^2<\infty$ is sufficient
for $\me_{\gamma} S_{\tau} < \infty$ to hold. Now the finiteness
of
\begin{equation*}
\me_{\gamma} e^{\gamma S_1} ~=~ \varphi(\gamma)^{-1} ~<~    \infty,
\end{equation*}
implies the finiteness of $\me_{\gamma} (S_1^+)^2$, and the proof
of part (a) is complete.\newline (b) We only consider the case
when $X$ is non-lattice since the lattice case can be
treated similarly. Denote by $R_x := S_{\tau(x)}-x$ the overshoot.
Since $\me e^{a\tau(x)}=\me_\gamma e^{\gamma S_{\tau(x)}}$, we
have in view of the already proved part (a)
\begin{equation}\label{1}
\lix \me_\gamma e^{\gamma R_x}  ~=~ \frac{\me e^{a\tau}-1}{\gamma
\me_\gamma S_\tau}.
\end{equation}
By Theorem \ref{main1}, if $\me e^{aN(x)}<\infty$, then $\me
e^{a\tau(x)}<\infty$. Therefore, according to part (a), we have $0
< \me_\gamma S_\tau<\infty$. This implies (see, for instance,
\cite[Theorem 10.3 on p.\,103]{Gut2009}) that, as $x \to \infty$,
$R_x$ converges in distribution to a random variable $R_{\infty}$
satisfying
\begin{equation*}
\mmp_\gamma \{R_{\infty} \leq x\}
~=~ \frac{1}{\me_\gamma S_\tau} \int_0^x \, \mmp_\gamma\{S_\tau > y\}\,{\rm d}y,
\quad   x \geq 0.
\end{equation*}
In particular, under $\mmp_{\gamma}$, $e^{\gamma R_x}$ converges in distribution to $e^{\gamma R_{\infty}}$. Further,
\begin{equation*}
\me_\gamma e^{\gamma R_{\infty}}
~=~ \frac{1}{\me_\gamma S_\tau} \int_0^\infty e^{\gamma y}\mmp_\gamma\{S_\tau>y\}\,{\rm d}y
~=~ \frac{\me_\gamma e^{\gamma S_\tau}-1}{\gamma \me_\gamma S_\tau}
~=~ \frac{\me e^{a\tau}-1}{\gamma \me_\gamma S_\tau}.
\end{equation*}
Therefore, \eqref{1} can be rewritten as follows:
\begin{equation}    \label{2}
\lix \me_\gamma e^{\gamma R_x}  ~=~ \me_\gamma e^{\gamma R_{\infty}}.
\end{equation}
Now we invoke a variant of Fatou's lemma sometimes called Pratt's
lemma \cite[Theorem 1]{Pratt}. To this end, note that, by a
standard coupling argument, we can assume w.l.o.g.\ that $R_x \to
R_{\infty}$ $\mmp_{\gamma}$-a.s. From \eqref{eq:N_decomposed} we
infer that for $f(y) := \me e^{aN(y)}$, $y \in \mr$ we have
\begin{equation*}
f(x)    ~=~ \me e^{a N(x)}  ~=~ e^{-a} \me e^{a\tau(x)}f(-R_x)
            ~=~ e^{\gamma x} e^{-a} \me_\gamma e^{\gamma R_x}f(-R_x).
\end{equation*}
$f$ is an increasing function and, therefore, has only
countably many discontinuities. Hence
$e^{\gamma R_x} f(-R_x)$ converges $\mmp_\gamma$-a.s.\ to
$e^{\gamma R_{\infty}} f(-R_{\infty})$. Further,
\begin{equation*}
e^{\gamma R_x}f(-R_x)   ~\leq~  e^{\gamma R_x} f(0)
\end{equation*}
and $e^{\gamma R_x} f(0)$ converges $\mmp_\gamma$-a.s.\ to
$e^{\gamma R_\infty}f(0)$. Finally,
\begin{equation*}
\lix \me_\gamma e^{\gamma R_x} f(0) ~=~ \me_\gamma e^{\gamma
R_\infty} f(0).
\end{equation*}
Therefore the assumptions of Pratt's lemma are fulfilled and an application of the lemma yields
\begin{eqnarray*}
\lix e^{-\gamma x}f(x)
& = &
e^{-a} \lix \me_\gamma e^{\gamma R_x} f(-R_x)
~=~ e^{-a} \me_\gamma e^{\gamma R_{\infty}}f(-R_{\infty})   \\
& = &
\frac{e^{-a}}{\me_\gamma S_\tau} \int_0^\infty e^{\gamma y}f(-y) \mmp_\gamma\{S_\tau>y\} \, {\rm d}y    \\
& = &
\frac{e^{-a} \me_\gamma \int_0^{S_\tau} e^{\gamma y}f(-y) \, {\rm d}y}{\me_\gamma S_\tau}.
\end{eqnarray*}
(c) From \eqref{eq:rho(x)=n} and \eqref{eq:V(A)} (with $R$ replaced by
$a$ and $M = \inf_{k \geq 1} S_k$), we infer
\begin{eqnarray*}
\me e^{a\rho(x)}
& = &   \underset{(-\infty,x]}{\int} \! \mmp\{M>x-y\} \, V({\rm d}y)    \\
& = & V(x) \mmp\{M>0\} - \underset{(0,\infty)}{\int} \! V(x-y) \, \mmp\{M\in {\rm d}y\},
\quad   x \geq 0.
\end{eqnarray*}
Assume that $X$ is non-lattice and set $D_1 :=
\frac{e^{-a}}{\gamma \me X e^{-\gamma X}}$. It follows from
\eqref{212121} that $D_1 \in (0,\infty)$ and from
\cite[Theorem 2.2]{IksMei} that
\begin{equation}\label{333}
V(x)~\sim~ D_1e^{\gamma x}, \ \ x\to\infty.
\end{equation}
The latter implies that for any $\varepsilon>0$ there exists an
$x_0>0$ such that
\begin{equation*}
(D_1-\varepsilon)e^{\gamma y}   ~\leq~  V(y)    ~\leq~
(D_1+\varepsilon)e^{\gamma y}
\end{equation*}
for all $y\geq x_0$. Fix one such $x_0$. Then for all $x\geq x_0$,
\begin{eqnarray*}
(D_1-\varepsilon) \, e^{\gamma x} \!\!\!
\underset{(0,x-x_0]}{\int} \!\!\! e^{-\gamma y} \, \mmp\{M\in {\rm
d}y\} & \leq &
\underset{(0,x-x_0]}{\int} \!\!\! V(x-y) \, \mmp\{M\in {\rm d}y\}   \\
& \leq & (D_1+\varepsilon) \, e^{\gamma x} \!\!\!
\underset{(0,x-x_0]}{\int} \!\!\! e^{-\gamma y} \, \mmp\{M\in {\rm
d}y\},
\end{eqnarray*}
and $\int_{(x-x_0,\infty)} V(x-y) \mmp\{M\in {\rm d}y\} \in [0,
V(x_0)]$. Letting first $x \to \infty$ and then $\varepsilon\to 0$
we conclude that
\begin{equation*}
\lix e^{-\gamma x} \underset{(0,\infty)}{\int} \! V(x-y) \,
\mmp\{M\in {\rm d}y\} ~=~ D_1\me e^{-\gamma M} \1_{\{M>0\}}.
\end{equation*}
Together with \eqref{333} the latter yields
\begin{eqnarray*}
\me e^{a\rho(x)}~&\sim& ~ D_1\big(\mmp\{M>0\}-\me e^{-\gamma M}
\1_{\{M>0\}}\big)e^{\gamma x}\\ &=&~D_1 \big(1-\me e^{-\gamma
M^+}\big)e^{\gamma x}, \ \ x\to\infty.
\end{eqnarray*}
Under the present assumptions, the random walk $(S_n)_{n\geq 0}$
drifts to $+\infty$ a.s. Therefore, $\mmp\{M>0\}>0$ which implies
that $1-\me e^{-\gamma M^+} > 0$ and completes the proof in the
non-lattice case.

The proof in the lattice case is based on the lattice version of
\cite[Theorem 2.2]{IksMei} and follows the same path.
\end{proof}

\section{Examples} \label{sec:apps_and_exs}

In this section, retaining the notation of Section 1, we
illustrate the results of Theorem \ref{main1} and Theorem
\ref{main2} by three examples.

\begin{ex}[Simple random walk]
Let $1/2 < p < 1$ and $\mmp\{X=1\} = p = 1-\mmp\{X=-1\} =: 1-q$.
Then the Laplace transform $\varphi$ of $X$ is given by
$\varphi(t)=pe^{-t}+qe^t$ and $R = -\log (2\sqrt{pq})$. According
to \cite[Formula (3.7) on p.~272]{Feller} and \cite[Example
1]{Doney19892}, respectively,
\begin{equation*}
\mmp\{\tau=2n-1\}   ~=~ \frac{1}{2q}\frac{{2n \choose n}}{2^{2n}(2n-1)}(2\sqrt{pq})^{2n},
\ \mmp\{\tau=2n\} = 0,
\quad   n \in \mn;
\end{equation*}
\begin{equation*}
\mmp\{\rho=2n\}  ~=~ (p-q){2n \choose n}(pq)^n, \
\mmp\{\rho=2n+1\}=0, \quad   n \in \mn_0.
\end{equation*}
Stirling's formula yields
\begin{equation}\label{22}
\frac{{2n \choose n}}{2^{2n}} \sim \frac{1}{\sqrt{\pi n}}, \ \
n\to\infty,
\end{equation}
which implies that
\begin{equation*}
\me e^{R\tau} < \infty \quad   \text{and}  \quad \me e^{R\rho} =
\infty.
\end{equation*}
\end{ex}

\begin{ex}
Let $X \od Y_1-Y_2$ where $Y_1$ and $Y_2$ are independent r.v.'s
with exponential distributions with parameters $\alpha$ and
$\kappa$, respectively, $0<\alpha<\kappa$. Then $\varphi(t) = \me
e^{-tX} = \frac{\alpha \kappa}{(\alpha+t)(\kappa-t)}$ and $R=-\log
(\frac{4\alpha\kappa}{(\alpha+\kappa)^2})$. According to
\cite[Formula (8.4) on p.~193]{Feller2}, for $a \in (0,R]$,
\begin{equation*}\me
e^{a\tau}   ~=~
(2\alpha)^{-1}(\alpha+\kappa-\sqrt{(\alpha+\kappa)^2-4\alpha
\kappa e^a})
                            ~<~ \infty.
\end{equation*}
Further, for $n\in\mn_0$,
\begin{eqnarray*}
\mmp\{\rho=n\} & = &
\underset{(-\infty,0]}{\int} \mmp\{\inf_{k \geq 1} S_k > -x\} \, \mmp\{S_n \in {\rm d}x\}    \\
& = &
\underset{(-\infty,0]}{\int} \underset{(-x,\infty)}{\int} \mmp\{\inf_{k \geq 0} S_k > -x-y\}
\, \mmp\{S_1 \in {\rm d}y \}\, \mmp\{S_n \in {\rm d}x\}.
\end{eqnarray*}
According to \cite[Formula (5.9) on p.\,410]{Feller2},
\begin{equation*}
\mmp\{\inf_{k \geq 0} S_k > -x-y\} ~=~ \mmp\{\sup_{k \geq 0}(-
S_k)<x+y\} ~=~ 1-\frac{\alpha}{\kappa} e^{-(\kappa-\alpha)(x+y)}.
\end{equation*}
Note that $S_n$ has the same law as the difference of two
independent random variables with gamma distribution with
parameters $(n,\alpha)$ and $(n,\kappa)$, respectively, which
particularly implies that, for $x>0$, the density of $S_1$ takes
the form $\frac{\alpha \kappa e^{-\alpha x}}{\alpha + \kappa}$.
Thus\footnote{We do not claim that this formula is new, but we
have not been able to locate it in the literature.}, for
$n\in\mn$,
\begin{eqnarray*}
\mmp\{\rho=n\}
& = &
\underset{(-\infty,0]}{\int} \int_{-x}^{\infty} \left(1-\frac{\alpha}{\kappa} e^{-(\kappa-\alpha)(x+y)}\right)
\, \frac{\alpha \kappa e^{-\alpha y}}{\alpha + \kappa} \, {\rm d}y \, \mmp\{S_n \in {\rm d}x\}   \\
& = &
\frac{\kappa- \alpha}{\kappa} \underset{(-\infty,0]}{\int} e^{\alpha x}
\, \mmp\{S_n \in {\rm d}x\} \\
& = &
\frac{\kappa- \alpha}{\kappa} \int_0^{\infty} \int_0^t e^{\alpha(s-t)} \frac{\alpha^n s^{n-1} e^{-\alpha s}}{(n-1)!} \frac{\kappa^n t^{n-1} e^{-\kappa t}}{(n-1)!} {\rm d}s \, {\rm d}t \\
& = &
\frac{\kappa-\alpha}{\kappa}\frac{\alpha^n \kappa^n}{n!(n-1)!} \int_0^{\infty} t^{2n-1} e^{-(\alpha + \kappa) t} {\rm d}t  \\
& = &
\frac{\kappa-\alpha}{(\kappa+\alpha)^{2n}}\alpha^n\kappa^{n-1}
{2n-1 \choose n},
\end{eqnarray*}
and
\begin{equation*}
\mmp\{\rho=0\}=\frac{\kappa-\alpha}{\kappa}.
\end{equation*}
Hence,
\begin{equation*}
\me e^{R \rho} ~=~\frac{\kappa-\alpha}{\kappa}\bigg(1+ \sum_{n \geq
1}4^{-n}{2n-1\choose n}\bigg) ~=~ \infty,
\end{equation*}
since relation \eqref{22} implies that the summands are of order
$1/\sqrt{n}$, as $n \to \infty$.
\end{ex}

Finally, we point out an explicit form of distribution of $X$ for
which $\me e^{R\rho(x)}<\infty$ for every $x\geq 0$.
\begin{ex}
Fix $h > 0$ and take any probability law $\mu_1$ on $\mr$ such
that the Laplace-Stieltjes transform
$$\psi(t) ~:=~ \int_\mr e^{-tx}\mu_1({\rm d}x), \quad t \geq 0,$$
is finite for $0 \leq t \leq h$ and infinite for $t>h$, and the left derivative of $\psi$
at $h$, $\psi'(h)$, is finite and positive. For instance, one can
take
$$	\mu_1({\rm d}x) ~:=~ c e^{-h |x|}/(1+|x|^{r}){\rm d}x,
\quad x\in\mr	$$
where $r>2$ and $c:=\big(\int_\mr e^{-h|x|}(1+|x|^r)^{-1}{\rm d}x\big)^{-1}>0$.

Now choose $s$ sufficiently large such that $\psi'(h) < s
\psi(h)$. Then $\varphi(t) = e^{-st} \psi(t)$ is the
Laplace-Stieltjes transform of the distribution $\mu:=\delta_s\ast
\mu_1$. Let $X$ be a random variable with distribution $\mu$.
Plainly, $\varphi(t)$ is finite for $0 \leq t \leq h$ but infinite
for $t > h$. Furthermore,
\begin{equation*}
\varphi'(t) ~=~ e^{-st}(\psi'(t)-s\psi(t)), \quad |t|\leq h.
\end{equation*}
In particular, $\varphi'(h) < 0$ which, among other things,
implies that $R=-\log \varphi(h)$ and that $\gamma_0=h$.
Therefore, $\me Xe^{-\gamma_0 X}=-\varphi^\prime(h)>0$,
and by Theorem \ref{main1}, $\me e^{R \rho(x)} < \infty$
for all $x \geq 0$.
\end{ex}

\noindent
\author{Alexander Iksanov   \\
        Faculty of Cybernetics  \\
        National T.\ Shevchenko University of Kiev  \\
        01033 Kiev, Ukraine \\
        e-mail: iksan@unicyb.kiev.ua    \\
        \and    \\
        Matthias Meiners    \\
        Department of Mathematics   \\
        Uppsala University \\
        Box 480, 751 06 Uppsala, Sweden \\
        e-mail: matthias.meiners@math.uu.se}

\begin{thebibliography}{99}

\bibitem{Als} {\sc Alsmeyer, G.} (1991). Some relations between
harmonic renewal measures and certain first passage times. {\em
Stat. Prob. Lett.} {\bf 12}, 19--27.

\bibitem{BelMaks} {\sc Beljaev, Ju. K. and Maksimov, V. M.} (1963). Analytical properties of a
generating function for the number of renewals. {\em Theor.
Probab. Appl.} {\bf 8}, 108--112.

\bibitem{Borov} {\sc Borovkov, A.A.} (1962). New limit theorems in boundary-value problems
for sums of independent terms. {\em Siber. Math. J.} {\bf 3},
645--694.

\bibitem{ChowTeicher}{\sc Chow, Y.~S. and Teicher, H.} (1988). Probability theory:
independence, interchangeability, martingales. Springer Texts in
Statistics. Springer-Verlag, New York.

\bibitem{Doney1980} {\sc Doney, R.A.} (1980). Moments of ladder heights in random walks.
{\em J. Appl. Prob.} {\bf 17}, 248--252.

\bibitem{Doney1989} {\sc Doney, R.A.} (1989). On the asymptotic
behaviour of first passage times for transient random walk. {\em
Probab. Theory Relat. Fields}. {\bf 81}, 239--246.

\bibitem{Doney19892} {\sc Doney, R.A.} (1989). Last exit times for random walks. {\em
Stoch. Proc. Appl.} {\bf 31}, 321--331.

\bibitem{Feller} {\sc Feller, W.} (1968). An introduction to
probability theory and its applications, Vol.~1, 3rd edition, John
Wiley \& Sons, New York etc.

\bibitem{Feller2} {\sc Feller, W.} (1971). An introduction to
probability theory and its applications, Vol.~2, 2nd edition, John
Wiley \& Sons, New York etc.

\bibitem{Gut2009} {\sc Gut, A.} (2009). Stopped random walks: Limit
theorems and applications, 2nd edition, Springer: New York.

\bibitem{Heyde1964} Heyde, C.C. (1964). Two probability theorems and their applications to some
first passage problems. {\em J. Austral. Math. Soc. Ser. B}. {\bf
4}, 214--222.


\bibitem{IksMei} {\sc Iksanov, A. and Meiners, M.} (2010). Exponential rate of almost sure convergence of intrinsic martingales in supercritical branching random walks. {\em J. Appl.
Prob.} {\bf 47}, to appear.

\bibitem{KestMaller} {\sc Kesten, H. and Maller,R.A.} (1996). Two
renewal theorems for general random walks tending to infinity.
{\em Probab. Theory Relat. Fields}.{\bf 106}, 1--38.

\bibitem{Pratt}{\sc Pratt J.~W.} (1960). On interchanging limits and
integrals. {\em Ann. Math. Stat.} {\bf 31}, 74--77.

\end{thebibliography}
\end{document}